\documentclass[12pt]{amsart}

\usepackage{amsmath,amscd,amssymb,amsfonts,enumerate,hhline,color, dsfont}

\definecolor{hot}{RGB}{65,105,225}

\usepackage[all]{xy}

\def\Supp{{\rm{Supp}\,}}

\newcommand{\bC}{{\mathbb C}}

\newcommand{\bR}{{\mathbb R}}

\newcommand{\bZ}{{\mathbb Z}}

\newcommand{\cV}{{\mathcal V}}

\newcommand{\Hom}{{\rm{Hom}}}

\newcommand{\Ab}{{\rm Ab}}
\newcommand{\Alex}{{\rm Alex}}
\newcommand{\TFM}{{\rm TFM}}

\theoremstyle{plain}
\newtheorem{thm}[subsection]{Theorem}
\newtheorem{prop}[subsection]{Proposition}
\newtheorem{lem}[subsection]{Lemma}
\newtheorem{cor}[subsection]{Corollary}

\newtheorem{df}[subsection]{Definition}

\theoremstyle{definition}
\newtheorem{rem}[subsection]{Remark}

\title[Arbitrary cohomology jump loci]{Examples of topological spaces with arbitrary cohomology jump loci}
\author{Botong Wang}
\address{Department of Mathematics,
University of Notre Dame, 255 Hurley Hall, IN 46556, USA} \email{bwang3@nd.edu}

\begin{document}

\begin{abstract}
Given any subvariety of a complex torus defined over $\bZ$ and any positive integer $k$,  we construct a finite CW complex $X$ such that the $k$-th cohomology jump locus of $X$ is equal to the chosen subvariety, and the $i$-th cohomology jump loci of $X$ are trivial for $i<k$. 
\end{abstract}

\maketitle


\section{Introduction}
The motivation of this note is to explore the consequence of the results of \cite{S1} and \cite{BW}, that the cohomology jump loci of a smooth complex quasi-projective variety are unions of torsion translates of subtori. 

Given a topological space $M$, we denote by $L(M)$ the space of rank one local systems on $M$, which is equal to $\Hom(\pi_1(M), \bC^*)$ and has a group structure. When $\pi_1(M)$ is finitely generated, $L(M)$ has a complex variety structure which is isomorphic to the cartesian product of $(\bC^*)^{b_1(M)}$ and a discrete finite abelian group. In $L(M)$, there are some canonically defined subvarieties called the cohomology jump loci (or characteristic varieties). They are defined to be $\Sigma^i_r(M)=\{\rho\in L(M)\;|\: \dim H^i(M, L_\rho)\geq r\}$, where $L_\rho$ is the rank one local system on $M$ associated to the representation $\rho$. They are always varieties defined over $\bZ$. When $r=1$, we omit $r$ and just write $\Sigma^i(M)$. 

The main result of this note is the following. 

\begin{thm}\label{main}
Fix a positive integer $n$. Let $Z$ be any (not necessarily irreducible) subvariety of $(\bC^*)^n$, which is defined over $\bZ$, and let $k$ be any positive integer. There exists a finite CW complex $M$, such that $L(M)=(\bC^*)^n$, $\Sigma^k(M)\cup \{\mathds{1}\}=Z\cup \{\mathds 1\}$, and $\Sigma^i(M)$ is either empty or equal to $\{\mathds 1\}$ for $i<k$. 
\end{thm}

We prove the theorem in the rest of this note. The proof is divided into two parts, $k=1$ and $k\geq 2$. When $k=1$, this is a group theoretic problem, and is essentially known (for example, \cite{SYZ} Lemma 10.3). For completeness, we include the proof still. In fact, the examples for the case $k\geq 2$ are inspired by the case $k=1$. 

When $k\geq 2$ and $n$ is even, and when $Z$ is not a union of torsion translates of subtori, our examples are homotopy $(k-1)$-equivalent to an abelian variety, but not homotopy $k$-equivalent to any quasi-projective variety. This shows that the result of \cite{BW} puts genuine higher homotopy obstruction to the possible homotopy types of quasi-projective varieties. 

Carlos Simpson has kindly pointed out to us that the construction for $k\geq 2$ has appeared in \cite{S2}. Thus, this note is a self contained proof of some essentially known result. 

\section{Constructing the examples for $k\geq 2$}
We assume $k\geq 2$ though out this section. 

Under the notation of Theorem \ref{main}, we start with a real torus $M_0=(S^1)^n$. Let $N_0$ be the universal cover of $M_0$, and let the covering map be $p_0: N_0\to M_0$. Fixing an origin $O\in M_0$, we attach a $k$-sphere $S^k$ to $M_0$ at $O$, obtaining a new space $M_1$. In other words, $M_1$ is the wedge sum of $M_0$ and $S^k$. We call such a $k$-sphere. Let $p_1: N_1\to M_1$ be the universal covering map. Then $N_1$ is obtained from $N_0$ by attaching infinitely many $k$-spheres parametrized by $\bZ^n$. Denote the $k$-sphere in $N_1$ corresponding to $J\in \bZ^n$ by $B_J$.

Suppose the subvariety $Z$ of $(\bC^*)^n$ is defined by Laurent polynomials \\$f_i(x_1, x_1^{-1}, \ldots, x_n, x_n^{-1})$, for $1\leq i\leq r$, where $x_j$'s are the coordinates of $(\bC^*)^n$. Suppose $f_i(x_1, \ldots, x_n)=\sum_{J\in \Lambda_i} a_i^J x^J$, where $\Lambda_i$ is a finite subset of $\bZ^n$, $a_i^J$ are integers and $x^J$ is the product of $x_j$'s under the usual multi-index notation. For example, if $n=2$ and $J=(-2, 3)$, then $x^J=x_1^{-2}x_2^3$. 

The CW complex $M$ satisfying the property of Theorem \ref{main} will be obtained by gluing $r$ $(k+1)$-cells to $M_1$ corresponding to the Laurent polynomials $f_1, \ldots, f_r$. To explain how to attach the $(k+1)$-cell corresponding to one Laurent polynomial $f_i$, it is easier to work on $N_1$. Fix a $J_0\in \bZ^n$. We can attach a $(k+1)$-cell $e^{k+1}_{J_0}$ to $N_1$, such that $\partial e^{k+1}_{J_0}$ represents the cycle $\sum_{J\in \Lambda_1} a_i^J B_{J_0+J}$ in $H_k(N_1, \bZ)$. Denote the new space by $N_{2, J_0}$. Notice that $N_1$ is $(k-1)$-connected. Hence $N_{2, J_0}$ is uniquely determined up to homotopy. Define $N_2$ to be the coproduct of $N_{2, J_0}$ over $N_1$ for all $J_0\in \bZ^n$. Then $N_2$ is obtained from $N_1$ by attaching infinitely many $(k+1)$-cells parametrized by $\bZ^n$. Suppose we attach these $(k+1)$-cells in a compatible way. Then the Galois action on $N_1$ by $\bZ^n$ extends to $N_2$. Now, let $M_2$ be the quotient of $N_2$ by the Galois action $\bZ^n$. By our construction, $M_2$ is obtained by attaching one $(k+1)$-cell to $M_1$. In the same way, we can attach more $(k+1)$-cells corresponding to other Laurent polynomials $f_2, \ldots, f_r$. Let $M$ be the space we obtained this way, and let $N$ be the cover of $M$ with Galois group $\bZ^n$. Denote the covering map by $p: M\to N$. 

By our construction, there is a natural isomorphism $\pi_1(M)\cong \pi_1(M_0)$. Since $M_0=(S^1)^n$, there is a natural isomorphism $L(M)\cong (\bC^*)^n$, and the isomorphism is induced by an isomorphism of the underlying scheme over $\bZ$. 

\begin{prop}
Under the above isomorphism, we have the following. 
\begin{enumerate}
\item $\Sigma^i(M)=\{\mathds 1\}$ for $0\leq i\leq \min\{n, k-1\}$, $\Sigma^i(M)=\emptyset$ for $n<i\leq k-1$. 
\item $\Sigma^k(M)=Z\cup \{\mathds 1\}$ when $k\leq n$, and $\Sigma^k(M)=Z$ when $k>n$. 
\item $M$ is homotopy $(k-1)$-equivalent to the real torus $(S^1)^n$. 
\item $M$ is not homotopy $k$-equivalent to any quasi-projective variety. 
\end{enumerate}
\end{prop}
The statements (1) and (3) are obvious from the construction. Moreover, (4) follows from (2) and \cite{BW}. To prove this, one has to argue that the cohomology jump locus $\Sigma^i$ is a homotopy $i$-equivalence invariant. We leave this to the reader. We will prove statement (2) in the rest of this section. 

We will compute the cohomology jump loci of $M$ via Alexander modules. Recall that $p: N\to M$ is the covering map with Galois action by $\bZ^n$. Then in a natural way, the homology groups $H_i(N, \bZ)$ become $\bZ^n$-modules. Notice that $\bZ^n$ is naturally identified with $H_1(M, \bZ)$, and in a natural way $L(M)=\Hom(H_1(M, \bZ), \bC^*)\cong (\bC^*)^n$. Thus, the group ring of $\bZ^n$ is naturally identified with $\bZ[x_1, x_1^{-1}, \ldots, x_n, x_n^{-1}]$, which is the coordinate ring of the underlying $\bZ$-scheme of $(\bC^*)^n$. As $\bZ^n$-module, $H_i(N, \bZ)$ has now a natural $\bZ[x_1, x_1^{-1}, \ldots, x_n, x_n^{-1}]$-module structure. These $H_i(N, \bZ)$, together with the $\bZ[x_1, x_1^{-1}, \ldots, x_n, x_n^{-1}]$-module structures are called Alexander modules. 

Denote the support of the Alexander modules $H_i(N, \bZ)$ in $(\bC^*)^n$ by $\cV^i(M)$. Then the cohomology jump loci and the support of the Alexander modules are closely related by the following theorem of Papadima and Suciu.

\begin{thm}[\cite{PS} Theorem 3.6]\label{psthm}
$$\bigcup_{i=0}^l \Sigma^i(M)=\bigcup_{i=0}^l \cV^i(M)$$
for any integer $l\geq 0$. 
\end{thm}

Since $M$ is homotopy $(k-1)$-equivalent to the real torus $(S^1)^n$, $N$ is homotopy $(k-1)$-equivalent to a point. Therefore, $\cV^0(M)=\{\mathds 1\}$ and $\cV^i(M)=\emptyset$ for $1\leq i\leq k-1$. According to Theorem \ref{psthm}, $\Sigma^i(M)\subset \{\mathds 1\}$. 

By construction, $H_k(N_1, \bZ)$ is a free $\bZ[x_1, x_1^{-1}, \ldots, x_n, x_n^{-1}]$-module of rank one. Recall that $N$ is obtained from $N_1$ by attaching many $(k+1)$-cells. Each $(k+1)$ gives a relation in $H_k(N_1, \bZ)$ as $\bZ$-module. After attaching infinitely many cells, $H_k(N_2, \bZ)$ is isomorphic to $\bZ[x_1, x_1^{-1}, \ldots, x_n, x_n^{-1}]/(f_1)$ as $\bZ[x_1, x_1^{-1}, \ldots, x_n, x_n^{-1}]$-modules. This can be proved by a standard Mayer-Vietoris sequence argument. We leave this to the reader. Similarly, $H_k(N, \bZ)\cong \bZ[x_1, x_1^{-1}, \ldots, x_n, x_n^{-1}]/(f_1, \ldots, f_r)$. Thus, $\cV^k(M)=Z$. Now, according to Theorem \ref{psthm}, $\Sigma^k(M)\cup \{\mathds 1\}=Z\cup \{\mathds 1\}$. 

To check whether $\mathds 1\in \Sigma^k(M)$ is to compute whether $H^k(M, \bC)=0$. This can be done directly from the construction. In fact, $H^k(M, \bC)\neq 0$ when $k\leq n$ and $H^k(M, \bC)=0$ when $k>n$. Thus we have proved the proposition. 

\begin{rem}
One can try to use the same construction for $k=1$. It works except that the circle (or 1-shpere) we attach may create extra elements in $H_1(M, \bZ)$. In fact, it works when $Z$ does not contain any torsion point. 
\end{rem}

\section{Group theoretic first cohomology jump loci}
It is an easy and well known fact that for a topological space $M$ and $\rho\in L(M)$, $H^1(M, L_\rho)\cong H^1(\pi_1(M), \bC_\rho)$. Here the representation $\rho: \pi_1(M)\to \bC^*$ gives $\bC$ a $\pi_1(M)$-module structure, and we emphasize the $\pi_1(M)$-module structure on $\bC$ by writing $\bC_\rho$. Therefore, the first cohomology jump loci of $M$ only depend on $\pi_1(M)$. 

\begin{df}
Let $G$ be a finitely presented group. We denote the character variety $\Hom(G, \bC^*)$ by $L(G)$, and define the group theoretic cohomology jump loci $\Sigma^i(G)=\{\rho\in L(G)\;|\; H^i(G, \bC_\rho)\neq 0\}$. 
\end{df}

The isomorphism between the cohomology of local system and the group cohomology shows that the natural isomorphism $L(M)\cong L(\pi_1(M))$ induces an isomorphism between subvarieties $\Sigma^1(M)\cong \Sigma^1(\pi_1(M))$. Given a finitely presented group $G$, there is a standard process to construct a finite CW complex with fundamental group $G$. Therefore, the case $k=1$ of Theorem \ref{main} is equivalent to the following.

\begin{prop}\label{gp}
Let $Z$ be any (not necessarily irreducible) subvariety of $(\bC^*)^n$ defined over $\bZ$. There is a finitely presented group $G$ such that $L(G)=(\bC^*)^n$ and $\Sigma^1(G)=Z\cup \{\mathds 1\}$. 
\end{prop}

Before proving the Proposition, we give an algorithm to compute the first group theoretic cohomology jump loci, which is slightly different from \cite{SYZ}. 

$H^1(G, \bC_\rho)$ can be computed by the quotient of 1-cycles by 1-boundaries, i.e., $H^1(G, \bC_\rho)=Z^1(G, \bC_\rho)/B^1(G, \bC_\rho)$, where
$$Z^1(G, \bC_\rho)=\{\tau\in \Hom_{\textsl{set}}(G, \bC)\,|\, \tau(ab)=\rho(a)\tau(b)+\tau(a)\;\textsl{for any}\; a, b\in G\}$$
 and
 $$B^1(G, \bC_\rho)=\{\tau\in \Hom_{\textsl{set}}(G, \bC)\,|\, \tau(a)=\rho(a)\tau(1)-\tau(1) \textsl{for any}\; a\in G\}.$$
 Denote the commutator subgroup of $G$ by $G'$ and denote the abelianization of $G$ by $\Ab(G)$, then a straightforward computation shows the following. 

\begin{lem}\label{mle}
When $\rho\neq \mathds 1$, 
$$H^1(G, \bC_\rho)\cong \{\tau\in \Hom(G', \bC)\,|\, \tau(sa)=\rho(s)\tau(a) \;\textsl{for any}\; s\in G, a\in G'\}.$$
\end{lem}

Given a finitely presented group $G$, we define an element $g\in G$ to be in the ``torsion free metabelian kernel", if $g\in G'$, and the image of $g$ in $\Ab(G')$ is torsion. It is very easy to check that the torsion free metabelian kernel forms a normal subgroup. We define the quotient of $G$ by its torsion free metabelian kernel to be the torsion free metabelianization of $G$, denoted by $\TFM(G)$. The following is a direct consequence of the previous lemma. 

\begin{cor}
The first cohomology jump locus of $G$ only depends on $\TFM(G)$. More preciously, the natural isomorphism $L(G)\cong L(\TFM(G))$ induces an isomorphism $\Sigma^1(G)\cong \Sigma^1(\TFM(G))$. 
\end{cor}

Given a finitely presented group $G$, we define the first Alexander module of $G$ to be the kernel of the natural map $\TFM(G)\to \Ab(G)$, or equivalently the commutator subgroup of $\TFM(G)$. Denote the first Alexander module of $G$ by $\Alex(G)$. Then we have a short exact sequence of groups,
$$0\to \Alex(G)\to \TFM(G)\to \Ab(G)\to 0.$$
This short exact sequence induces an action of $\Ab(G)$ on $\Alex(G)$, which gives $\Alex(G)$ a $\bZ[\Ab(G)]$-module structure. As an analog of Theorem \ref{psthm}, the following follows from Lemma \ref{mle}.

\begin{cor}\label{lcor}
Suppose $b_1(G)>0$. The coordinate ring of the underlying $\bZ$-scheme of $L(G)$ is naturally isomorphic to $\bZ[\Ab(G)]$. Under this isomorphism, 
$$\Sigma^1(G)=\Supp(\Alex(G))\cup \{\mathds 1\}.$$
\end{cor}

\begin{proof}[Proof of Proposition \ref{gp}]
Now, we are ready to construct the example satisfying the condition in the proposition. As before, we assume the defining equations of $Z$ to be $f_1, \ldots, f_r$, where $f_i\in \bZ[x_1, x_1^{-1}, \ldots, x_n, x_n^{-1}]$. 

We start with $G_0$, which we define to be the free group of $n$ generators $g_1, \ldots, g_n$. Let $G_1$ be the torsion free metabelianization of $G_0$. The choice of the generators of $G$ gives a natural isomorphism $\Ab(G_1)\cong (\bZ)^n$. Let $x_1, \ldots, x_n$ be the natural coordinates on $(\bZ)^n$. Then $\Alex(G_1)$ has a natural $\bZ[x_1, x_1^{-1}, \ldots, x_n, x_n^{-1}]$-module structure. Since $G_0$ is the fundamental group of $n$-loops, we can realize $\Alex(G_1)\cong \Alex(G_0)$ as the first homology group of the integral ``net" $N\subset \bR^n$. Here $N=\bigcup_{1\leq i\leq n}(\bZ^n+l_i)$, where $l_i$ is the $i$-th coordinate axis. $\bZ^n$ acts on $N$ by translation. Hence $\bZ^n$ also acts on $\Alex(G_1)\cong H_1(N, \bZ)$. In fact, the $\bZ[x_1, x_1^{-1}, \ldots, x_n, x_n^{-1}]$-module structure on $\Alex(G_1)$ can be interpreted this way. 

As a $\bZ[x_1, x_1^{-1}, \ldots, x_n, x_n^{-1}]$-module, $H_1(N, \bZ)$ is generated by the unit squares in each coordinate planes, with a proper choice of orientation. Denote these squares by $\gamma_{ij}$, $1\leq a<b \leq n$. More canonically, we will allow $a\geq b$ and put the condition $\gamma_{ab}+\gamma_{ba}=0$. Each unit cubic induces a relation between these generators. In fact, $H_1(N, \bZ)\cong \bigoplus_{1\leq a, b\leq n}/I$, where $I$ is the ideal generated by elements $\gamma_{ab}+\gamma_{ba}$ and $(x_a-1)\gamma_{bc}+(x_b-1)\gamma_{ca}+(x_c-1)\gamma_{ab}$ for all $1\leq a, b, c \leq n$. It is not hard to see that as a $\bZ[x_1, x_1^{-1}, \ldots, x_n, x_n^{-1}]$-module, $\Alex(G_1)$ has rank $n-1$. 

Since $H_1(N, \bZ)$ is a $\bZ[x_1, x_1^{-1}, \ldots, x_n, x_n^{-1}]$-module, for each $1\leq i\leq r$ and $1\leq a, b \leq n$, $h^i_{ab}\stackrel{\textrm{def}}{=}f_i(x_1, , x_1^{-1}, \ldots, x_n, x_n^{-1})\gamma_{ab}$ is a well-defined element in $H_1(N, \bZ)$. By our construction, the isomorphism $\Alex(G_1)\cong H_1(N, \bZ)$ sends $g_ag_bg_a^{-1}g_b^{-1}$ to $\gamma_{ab}$. Denote the element in $\Alex(G_1)$ corresponding to $h^i_{ab}$ by $\tilde{h}^i_{ab}$. Let $H_2$ be the normal subgroup of $G_2$ generated by $\tilde{h}^i_{ab}$, for all $1\leq i\leq r$ and $1\leq a, b\leq n$. Now, define $G_3$ to be the quotient $G_2/H_2$. Then $\Ab(G_3)=\Ab(G_2)$, and by construction, we have an isomorphism of $\bZ[x_1, x_1^{-1}, \ldots, x_n, x_n^{-1}]$-modules,
$$\Alex(G_3)\cong \Alex(G_2)\otimes_{\bZ[x_1, x_1^{-1}, \ldots, x_n, x_n^{-1}]} \bZ[x_1, x_1^{-1}, \ldots, x_n, x_n^{-1}]/(f_1, \ldots, f_r).$$
Since $\Alex(G_2)$ is of rank $n-1$ as a $\bZ[x_1, x_1^{-1}, \ldots, x_n, x_n^{-1}]$-module, $\Supp (\Alex(G_2))=L(X)$. Therefore, $\Supp(\Alex(G_3))$ is the subvariety defined by $(f_1, \ldots, f_r)$, that is $Z$. 

Finally, by Corollary \ref{lcor}, $\Sigma^1(G_3)=Z\cup\{\mathds 1\}$. Thus we have finished the proof of the proposition. 
\end{proof}

\begin{rem}
For a general connected finite CW-complex $X$, $\Hom(\pi_1(X), \bC^*)$ is a direct product of a complex torus with a finite abelian group. So one can ask a more general question, ``for a finite abelian group $A$, and any (not necessarily irreducible) variety $Z\subset (\bC^*)^n\times A$, does Theorem \ref{main} hold?" The answer is yes, and we briefly describe how to adjust our previous constructions to this situation.

First, we can write the coordinate ring of $\Hom(\pi_1(X), \bC^*)$ as 
$$R\stackrel{\textrm{def}}{=}\bZ[x_1, x_1^{-1}, \ldots, x_n, x_n^{-1}, y_1, \ldots, y_m]/(y_1^{l_1}, \ldots, y_m^{l_m}).$$
Suppose $Z$ is defined by ideal $I=(f_1, \ldots, f_r)$ of $R$. 

In the case of $k\geq 2$, we start with $M_0=(S^1)^n\times Y$, where $Y$ is an Eilenberg-Maclane space of type $K(A, 1)$. Attaching a $k$-sphere to $M_0$, we obtain $M_1$. Then, similar to the previous construction in section 2, we can attach $r$ $(k+1)$-cells to $M_1$ corresponding to the functions $f_1, \ldots, f_r$. Denote the resulting space by $M_2$. Now,  it follows from Theorem \ref{psthm} and the same argument as before that $M_2$ satisfies the cohomology jump loci property in Theorem \ref{main}. Notice that $Y$ can be chosen to be a finite-type CW-complex, but it may not be finite. Therefore, even though $M_2$ may not be a finite CW-complex, it is of finite type. However, by replacing $M_2$ by its $k'$-skeleton, where $k'$ is sufficiently large, it serves as the example of Theorem \ref{main}. 

In the case of $k=1$, again it suffices to talk about group theoretic cohomology jump loci. We start with the group 
$$G_0=<g_1, \ldots, g_n, h_1, \ldots, h_m\,|\, h_1^{l_1}=1, \cdots, h_m^{l_m}=1>.$$
Let $G_1=\TFM(G_0)$. Then $\Alex(G_1)=\Alex(G_0)$ is a finitely generated $R$-module, locally with positive rank. Similar to what we did earlier in this section, by taking a quotient of $G_1$ by the subgroup corresponding to $I\cdot \Alex(G_1)$, we obtain a group $G_2$. Then $\Alex(G_2)\cong \Alex(G_1)\otimes_R R/I$. According to Corollary \ref{lcor}, we have 
$$\Sigma^1(G_2)=Z\cup \{\mathds 1\}.$$
\end{rem}

\end{document}